\documentclass[11pt]{amsart}
\usepackage{amsmath}
\usepackage{amsfonts}
\usepackage{float}
\usepackage{mathrsfs}
\usepackage{mathtools}
\usepackage{latexsym}
\usepackage{amsmath,amsthm,amssymb, amscd,color}
\usepackage{mathrsfs}%For mathscr
\usepackage[all]{xy}
\usepackage[all]{xy}
\usepackage{graphicx}
\usepackage{resizegather}
\usepackage{tikz}
\usepackage{tikz-cd}
\usepackage{breqn}
\usepackage{cleveref}
\usepackage{enumerate}

\usepackage{blkarray} 
\usepackage{framed}

\newtheorem{innercustomgeneric}{\customgenericname}
\providecommand{\customgenericname}{}
\newcommand{\newcustomtheorem}[2]{%
	\newenvironment{#1}[1]
	{%
		\renewcommand\customgenericname{#2}%
		\renewcommand\theinnercustomgeneric{##1}%
		\innercustomgeneric
	}
	{\endinnercustomgeneric}
}

\newcustomtheorem{customthm}{Theorem}
\newcustomtheorem{customlemma}{Lemma}
\newcustomtheorem{customprop}{Proposition}

\theoremstyle{plain}
\newtheorem{theorem}{Theorem}[section]
\newtheorem{lemma}[theorem]{Lemma}

\numberwithin{equation}{section}

\theoremstyle{definition}
\newtheorem{definition}[theorem]{Definition}

\newtheorem{proposition}[theorem]{Proposition}
\newtheorem{remark}[theorem]{Remark}

\theoremstyle{remark}

\newcommand{\RR}{\mathbb{R}}

\newcommand{\NN}{\mathbb{N}}

\newcommand{\SN}{\mathbb{S}}

\begin{document}

\title[Dynamics and integrability of polynomial vector fields on the $n$-sphere]{Dynamics and integrability of polynomial vector fields on the $n$-dimensional sphere}
%\author{Joji Benny and Soumen Sarkar}
%\address{Department of Mathematics, Indian Institute of Technology Madras, India}

\author[S. Jana]{Supriyo Jana}
\address{Department of Mathematics, Indian Institute of Technology Madras, India}
\email{supriyojanawb@gmail.com}

\author[S. Sarkar]{Soumen Sarkar}
\address{Department of Mathematics, Indian Institute of Technology Madras, India}
\email{soumen@iitm.ac.in}

\date{\today}
\thanks{}

\abstract 
In this paper, we characterize arbitrary polynomial vector fields on $\SN^n$. We establish a necessary and sufficient condition for a degree one vector field on the odd-dimensional sphere $\SN^{2n-1}$ to be Hamiltonian. Additionally, we classify polynomial vector fields on $\SN^n$ up to degree two that possess an invariant great $(n-1)$-sphere. We present a class of completely integrable vector fields on $\SN^n$. We found a sharp bound for the number of invariant meridian hyperplanes for a polynomial vector field on $\SN^2$. Furthermore, we compute the sharp bound for the number of invariant parallel hyperplanes for any polynomial vector field on $\SN^n$. Finally, we study homogeneous polynomial vector fields on $\SN^n$, providing a characterization of their invariant $(n-1)$-spheres.
\endabstract

\maketitle

\section{Introduction}
Let  $P_1,\ldots,P_d$ be polynomials in $\mathbb{R}[x_1,\ldots,x_d]$. Then, the following system of differential equations
\begin{equation} \label{eq: I1}
 \frac{dx_i}{dt} = P_i(x_1, \ldots,x_d)       \end{equation}
for $i = 1,\ldots, d$ is called a polynomial differential system in $\mathbb{R}^d$. The differential operator
\begin{equation}  \label{eq: I2}
 \chi = \sum_{i=1}^d P_i \frac{\partial}{\partial x_i}    
\end{equation}
is called the vector field associated with the system \eqref{eq: I1}. The number $\max\limits_{1\leq i\leq d} \{ \deg P_i\}$ is called the degree of the polynomial vector field \eqref{eq: I2}. 

The polynomial vector field $\chi$ is called homogeneous if $P_1, \ldots, P_d$ are homogeneous polynomials of the same degree.

We note that the differential system \eqref{eq: I1} with $d=2$ has been studied since 1900. In the same year, Hilbert proposed a problem related to this system which is now known as Hilbert's 16th problem, see for instance \cite{hilbert1900, Ila02}.
 
A subset of $\RR^d$ is called invariant for the vector field \eqref{eq: I2} if the set is a union of some trajectories of the vector field. An algebraic set given by the zero set of some polynomial $f \in \mathbb{R}[x_1,x_2, \ldots,x_d]$ is called an invariant algebraic set if there exists $K\in \RR[x_1,x_2, \ldots,x_d]$ such that $\chi f = Kf$. In such case, $\chi$ is called a polynomial vector field on the hypersurface $\{f=0\}$ and $K$ is called the cofactor of $\chi$ for $\{f=0\}$.

Hilbert's 16th problem \cite{hilbert1900} asks bound for the number of invariant isolated curves in terms of the degree of polynomial vector fields in $\RR^2$. Since then, several researches have been done on this topic, see \cite{Ila02} and the references therein. However, the study of polynomial vector fields in higher dimensions started blooming in the 1970s, see for instance \cite{CL90, Go69, Ho79, Jou79} and several papers by J. Llibre and his collaborators. One natural question can be asked after the above definition. How many invariant algebraic sets are there for a given vector field? Another somewhat related question is whether a polynomial system has a first integral. This question has been studied since the work of Darboux \cite{Dar78} and Poincar\'{e} \cite{H1881}, and still, many interesting questions have remained to be answered. Darboux \cite{Dar78} showed the computation of (rational) first integral when the polynomial vector field in $\RR^2$ possesses plenty of invariant algebraic curves. Jouanolou \cite{Jou79} extended similar results in higher dimensions. Some more developments on this topic can be found in \cite{CL99,llibre2018darboux, LliZh09}. Also, 
generalizing the Darboux theory of integrability from planar vector fields, Llibre and Bola\~{n}os \cite{LliBol11} showed that polynomial vector fields on a regular algebraic hypersurface having a sufficient number of invariant algebraic hypersurfaces possess a rational first integral. Therefore, it is essential to study various invariant algebraic hypersurfaces for a given vector field.

We recall that Camacho  \cite{Cam} studied some properties of Morse-Smale vector fields on $\SN^2$  of degree two.
%a class of degree two homogeneous polynomial vector fields on $\SN^2$ has been studied in \cite{Cam}. 
The invariant circles for quadratic polynomial vector fields on $\mathbb{S}^2$ are discussed in \cite{llibre2006homogeneous} and \cite{LliZho11}. Also, the homogeneous polynomial vector fields on $\SN^2$ and their number of invariant (great) circles have been explored in \cite{LliPesII06}. 
The Darboux theory of integrability and the maximum number of invariant `parallels' and `meridians' for polynomial vector fields on $\SN^n$ have been studied in \cite{llibre2018darboux}. The authors and Benny \cite{BJS_24} gave a characterization of cubic polynomial vector fields on $\SN^2$ and studied a few properties of these vector fields. 

In this paper, we classify all polynomial vector fields of arbitrary degree on the standard $n$-dimensional sphere
$\SN^n :=\{(x_1,\ldots,x_{n+1})\in \RR^{n+1} ~|~ \sum\limits_{i=1}^{n+1}x_i^2-1=0\}.$
Then, we study their invariant meridian hyperplanes, parallel hyperplanes, and $(n-1)$-spheres. We also determine the integrability of polynomial vector fields on $\SN^n$ and find Hamiltonian polynomial vector fields on $\SN^n$.
We note that the intersection of a hyperplane $L:=\{\sum\limits_{i=1}^{n+1}a_ix_i+b=0\}$ with $\SN^n$ is said to be an $(n-1)$-sphere in $\SN^n$ where $a_i,b\in \RR$ with $a_1^2 + \cdots + a_{n+1}^2=1$ and $|b|<1$. The following definitions can be obtained as particular cases.
\begin{enumerate}
    \item If the hyperplane $L$ passes through the origin (i.e., $b=0$ in $L$), the intersection is called a great $(n-1)$-sphere in $\SN^n$.
    \item If $a_{n+1}=0$ and $b=0$ in $L$, the intersection is called a meridian and $L$ is called a meridian hyperplane.
    \item If $a_1,\ldots,a_n=0$ in $L$, the intersection is called a parallel and $L$ is called a parallel hyperplane.
\end{enumerate}

The paper is organized as follows. Section \ref{sec:basic_on_vf} contains some basic notions required in this paper. First, we recall the definition of the extactic polynomial of a given vector field associated with a vector subspace of the ring of polynomials. We recall definitions of first integral, complete integrability, and Hamiltonian vector field. We discuss the transformation of polynomial vector fields under the stereographic projection.

In Section \ref{sec:vf_on_sn}, we study some properties of the polynomial vector fields on $\SN^n$. 
 \begin{theorem}\label{thm:deg-n-vfld}
    Let $\chi=(P_1,...,P_{n+1})$ be a degree $m$ polynomial vector field in $\RR^{n+1}$. Then $\chi$ is a vector field on $\SN^n$ if and only if
    \begin{equation}\label{eq:deg-n-form}
        \begin{split}
            P_i=\sum\limits_{j=1}^{d(m)}(1-(\sum\limits_{k=1}^{n+1} x_k^2)^j) f_{ij} +\sum\limits_{j=1}^{n+1} A_{ij}x_j~\mbox{with}~d(m)=\begin{cases}
                \frac{m-1}{2}&~\mbox{if}~m~\mbox{is odd}\\
                \frac{m}{2}&~\mbox{if}~m~\mbox{is even}
            \end{cases}
        \end{split}
    \end{equation}
where $f_{ij}$ is a polynomial with degree less than or equals to $(m-2j)$ and
$A_{ij}$ is a polynomial with degree less than or equals to $(m-1)$ such that $A=(A_{ij})$ is a skew-symmetric matrix.
\end{theorem}
\noindent We show that if a linear polynomial is given, then there exist polynomial vector fields on $\SN^n$ having that linear polynomial as a first integral, see \Cref{thm:linear-fi}. Theorem \ref{thm:cmplt-int-vfld} states that 
there exist completely integrable degree $m$ vector fields on $\SN^n$ for each $m,n\in \NN$.
We study polynomial vector fields on $\SN^n$ of degree up to two, which have an invariant great $(n-1)$ sphere in $\SN^n$, see \Cref{thm:inv-grt-sphere-upto-2}. We give a necessary condition for a polynomial vector field on $\SN^{2n-1}$ to be Hamiltonian, see \Cref{thm:hamiltonian-necessary}. We give a necessary and sufficient condition for a degree one vector field on $\SN^{2n-1}$ to be Hamiltonian. 
\begin{theorem}\label{thm:deg-one-hamiltonian}
Suppose $\chi=(P_1,\ldots,P_{2n})$ is a degree one vector field on $\SN^{2n-1}$ such that $$\begin{pmatrix}
    P_1&\dots &P_{2n}
\end{pmatrix}^t=Ax$$ where $A=(a_{ij})_{2n\times 2n}$ is a constant skew-symmetric matrix and $x^t=\begin{pmatrix}
    x_1&\dots&x_{2n}
\end{pmatrix}$.
Then $\chi$ is Hamiltonian vector field if and only if $B=(b_{ij})_{2n\times 2n}$ with $b_{2i-1,j}=a_{2i,j}$ and $b_{2i,j}=-a_{2i-1,j}$ is a symmetric matrix.
\end{theorem}

In Section \ref{sec:inv_mer_par}, we discuss on the number of invariant meridian hyperplanes for a polynomial vector field on $\SN^n$, taking into account multiplicities.
\begin{theorem}\label{mer-bound-reach}
Let $m,n >1$. The following statements hold.
\begin{enumerate}
    \item There are vector fields of degree $m$ on $\SN^n$ with at least $\binom{n-1}{2}(m-2)+(n-1)(m-1)$ invariant meridian hyperplanes.
    \item There are vector fields of degree $m$ on $\SN^n$ with at least $(m-1)$ invariant meridian hyperplanes each with multiplicity $(n-1)$. 
    \item A degree $m$ vector field on $\SN^2$ having finitely many invariant meridian hyperplanes can have maximum $m$ invariant meridian hyperplanes. Moreover, this bound can be reached. 
    \item There are vector fields of degree $m$ on $\SN^3$ with $3m-2$ invariant meridian hyperplanes. 
\end{enumerate}
\end{theorem}
\noindent We prove that there are polynomial vector fields on $\SN^n$ having no meridian hyperplanes, see \Cref{thm:no-meridians}. Then, we compute the bound of invariant parallel hyperplanes.
\begin{theorem}\label{thm:parallel-bound}
Suppose $\chi$ is a polynomial vector field on $\SN^n$ of degree $m$ having finitely many invariant parallel hyperplanes. Then, $\chi$ has at most $m-1$ invariant parallel hyperplanes. Also, this bound is sharp.
\end{theorem}

In Section  \ref{sec:hom_vf}, first, we derive the equation of higher dimensional cones. We show that an $(n-1)$-sphere in $\SN^n$ is invariant if and only if the cone on the $(n-1)$-sphere is invariant, see Lemma \ref{lem:cone-invariant}. We classify all homogeneous polynomial vector fields on $\SN^n$, see \Cref{prop:hom-classify}. We prove the following.
\begin{theorem}\label{thm_sp_grtsp}
Suppose $\chi=(P_1,...,P_{n+1})$ is a homogeneous polynomial vector field on $\SN^n$ of degree $m\leq 2$. Then, every invariant $(n-1)$-sphere of $\chi$ is a great $(n-1)$-sphere.
\end{theorem}
\noindent We give a necessary and sufficient condition when a polynomial vector field on $\SN^n$ has an invariant $(n-1)$-sphere which is not a great sphere, see \Cref{prop:inv-sp-not-grt}. Then, we prove the following.
\begin{theorem}\label{thm_inv_grt_deg1}
   Suppose $\chi$ is degree one vector field on $\SN^n$ of the form \eqref{eq:deg-n-form}. Then $\{\sum\limits_{i=1}^{n+1}a_ix_i=0\}\cap \SN^n$ is an invariant great $(n-1)$-sphere of $\chi$  if and only if $\begin{pmatrix}
       a_1&\cdots&a_{n+1}
   \end{pmatrix}^t$ is an eigenvector of $A=(A_{ij})$ corresponding to the eigenvalue 0.
\end{theorem}
\noindent We also discuss invariant great $(n-1)$-sphere in $\SN^n$ for a homogeneous quadratic vector field. We denote $X^t:=(x_1,\ldots,x_{n+1})$ where $x_i\in \RR$ for $i=1,\ldots,n+1$.
\begin{theorem}\label{thm_inv_grt_deg2}
Let $\chi=(P_1, \ldots, P_{n+1})$ be a homogeneous quadratic vector field on $\SN^n$ with $P_i=X^t B_i X$ and $A_i:=[(B_1e_i) \cdots (B_{n+1}e_{i})]$ where each $B_i$ is a constant symmetric matrix. Then $\{{\bf a}^t X =0\} \cap \SN^n$ with ${\bf a}^t:=(a_1 \ldots a_{n+1})$ is an invariant great $(n-1)$-sphere if and only if $2\sum\limits_{i=1}^{n+1} a_i B_i = {\bf a}{\bf b}^t + {\bf b}{\bf a}^t$ for some ${\bf b}^t \in \RR^{n+1}$.

Moreover, if $\bf{a}$ is an eigenvector of each $A_i$ with eigenvalue $\lambda_i$, then ${\bf a}^t X=0$ produces an invariant great $(n-1)$-sphere with cofactor $\sum\limits_{i=1}^{n+1}\lambda_ix_i$. 
\end{theorem}

\section{Preliminaries}\label{sec:basic_on_vf}
In this section, we recall the concept of extactic polynomial for a polynomial vector field. Then, we describe how an extactic polynomial determines invariant algebraic sets. We recall some basic definitions. Lastly, we perform the computations to convert a polynomial vector field by the stereographic projection from the north pole.

In order to study invariant algebraic sets on an algebraic hypersurface in $\mathbb{R}^d$,  one may use the tool extactic polynomial. We briefly recall this concept following \cite{bolanos2013number}.
Let $W$ be a $k$-dimensional vector subspace of $\mathbb{R}[x_1,x_2,\dots,x_d]$ with a basis $\{     v_1,\dots,v_k\}$. Then the extactic polynomial of the vector field $\chi$ associated to $W$ is given by 
\begin{equation*}
\mathcal{E}_W(\chi) = \det \begin{pmatrix}
v_1 &  v_2 & \cdots &   v_k \\ \chi(v_1) & \chi(v_2) & \cdots & \chi(v_k) \\ \vdots & \vdots & \ddots  & \vdots \\ \chi^{k-1}(v_1) & \chi^{k-1}(v_2) & \cdots & \chi^{k-1}(v_k) \end{pmatrix},
\end{equation*}
 where $\chi^j(v_i) = \chi^{j-1}(\chi(v_i))$ for any $i,j$. We note that the definition of the extactic polynomial is independent of the choice of basis of $W$. 
  
   In this paper, we always count number of invariant algebraic set taking into account multiplicities. We recall the definition of the multiplicity of an algebraic set given by the zero set of a polynomial $f$.

\begin{definition}  \label{def: 21}
The algebraic set given by $f=0$ with $f \in W$ has multiplicity $m$ for $\chi$ if $\mathcal{E}_W(\chi) \neq 0$ and $f^m$ divides   $\mathcal{E}_W(\chi) $ and for $m' > m$, $f^{m'}$ is not a factor of $\mathcal{E}_W(\chi) $.  
 \end{definition}
 See more details regarding multiplicity in \cite{christopher2007multiplicity, LliZh09}. We shall use the following result whose proof can be found in \cite[Proposition 1]{LliMed07}.
 
 \begin{proposition} \label{Prop:1}
  Let $\chi$ be a polynomial vector field in $\mathbb{R}^d$ and $W$ a finite dimensional vector subspace of $\mathbb{R}[x_1,x_2, \dots,x_d]$ with $\dim(W) >1$. If $\{ f=0  \}$ is an invariant algebraic set for the vector field $\chi$ and $f \in W,$ then $f$ is a factor of $\mathcal{E}_W (\chi).$
  \end{proposition}

% We quote the following result from \cite{LliBol11}.

% \begin{proposition} \label{thm:1}
% Let $S$ be a regular algebraic hypersurface of degree $d$ in  $\mathbb{R}^{n+1}$. The polynomial vector field $\chi$ on $S$ of degree $m>0$ admits ${n+m \choose n+1} - {n+m-d \choose n+1} +n$ invariant algebraic hypersurfaces irreducible in $\mathbb{C}[x_1,x_2,\cdots,x_{n+1}]$ if and only if $\chi$ has a rational first integral.
% \end{proposition}

\begin{definition}
Let $U$ be an open subset of $\RR^d$. A non-constant smooth map $G \colon U \to \mathbb{R}$ is called a first integral of the vector field \eqref{eq: I2} on $U$ if $G$ is constant on all solution curves of the system \eqref{eq: I1} contained in $U$; i.e., $G(x_1(t),\ldots,x_d(t)) =$ constant for all values of $t$ for which the solution $(x_1(t), \ldots, x_d(t))$ is defined and contained in $U$.
\end{definition}
Note that the map $G$ is a first integral of the vector field \eqref{eq: I2} on $U$ if and only if $\chi G=0$ on $U$. The function $G$ is called the rational first integral when $G$ is a rational function. If the vector field \eqref{eq: I2} has a rational first integral, then the vector field possesses infinitely many invariant algebraic sets. A proof of this fact can be found on page 102 of \cite{Jou79}.

\begin{definition}
    The vector field \eqref{eq: I2} is called completely integrable in an open set $U$ if it has $(d-1)$ independent first integrals on $U$.
\end{definition}
If the vector field \eqref{eq: I2} is completely integrable with $(d-1)$ independent first integrals $G_1,\ldots, G_{d-1}$ then the orbits of the system are contained in $\bigcap\limits_{i=1}^{d-1}\{G_i=c_i\}$ for some $c_i\in \RR$. The complete integrability of vector fields on $\RR^d$ is discussed in \cite{LlRaRa20}.

\begin{definition}
    A vector field $\chi=(P_1,\ldots,P_{2n})$ is called Hamiltonian vector field if there exists a smooth map $H:\RR^{2n}\to \RR$ such that $P_{2i-1}=-\frac{\partial H}{\partial x_{2i}}$ and $P_{2i}=\frac{\partial H}{\partial x_{2i-1}}$ for $1\leq i\leq n$. The function $H$ is called a Hamiltonian for $\chi$.
\end{definition}
If $H$ is a Hamiltonian for $\chi$, then $H$ is a first integral of $\chi$. If $\chi$ has $n$ functionally independent Poisson commuting first integrals including the Hamiltonian $H$, then the differential system corresponding to $\chi$ can be solved by quadratures, see \cite[Chapter 1]{arutyunov2019elements}.

Next, we discuss the transformation of polynomial vector fields under the stereographic projection. We consider stereographic projection from the {\it north pole} $:={(0, \ldots, 0,1) \in \RR^{n+1}}$.  Note that the computation has been done for $n=2$ in \cite{LliZho11}. Here, we perform the computations for arbitrary $n>2$. We recall from the proof of \cite[Theorem 59.3]{Mun_2nded} that 
$$f \colon \SN^n\setminus \{(0,\ldots,0,1)\}\to \RR^n~\mbox{defined by}~ x=(x_1,...,x_{n+1})\mapsto u=(u_1,...,u_n)$$ where $u_i=\frac{x_i}{1-x_{n+1}}$ for $i=1,...,n$ is called the Stereographic projection from the north pole. The inverse map $$f^{-1}:\RR^n \to \SN^n\setminus\{0,\ldots,0,1\}~\mbox{is given by}~(u_1,...,u_n)\mapsto (x_1,...,x_{n+1})$$ where $x_i=\frac{2u_i}{|u|^2+1}$ for $i=1,...,n$ and $x_{n+1}=\frac{|u|^2-1}{|u|^2+1}$.

We note that $f$ is a diffeomorphism. Let $\chi=(P_1,...,P_{n+1})$ be a polynomial vector field on $\SN^n$ and the induced vector field be given by $f_*\chi=(Q_1,...,Q_n)$ on $\RR^n$. Then
$$\begin{pmatrix}
    Q_1 &\cdots &Q_n
\end{pmatrix}^T=\Big(\frac{\partial u_i}{\partial x_j}\Big)_{n\times (n+1)}\begin{pmatrix}
    P_1&\cdots&P_{n+1}
\end{pmatrix}^T.$$
For $i,j\in \{1,...,n\}$, we get $$ \frac{\partial u_i}{\partial x_{n+1}}=\frac{x_i}{(1-x_{n+1})^2}=\frac{|u|^2+1}{2}u_i \mbox{ and } \frac{\partial u_i}{\partial x_j}=\begin{cases}
    \frac{1}{1-x_{n+1}}=\frac{|u|^2+1}{2}&~\mbox{if}~i=j,\\
    0& ~\mbox{if}~i\neq j .
\end{cases}
$$ Hence, $$Q_i=\frac{|u|^2+1}{2}(P_i+u_iP_{n+1}).$$

Suppose that $\deg P_i\leq m$ for $i=1, \ldots, n+1$. We define 
$$\Tilde{P_i}:=(|u|^2+1)^mP_i\Big(\frac{2u_1}{|u|^2+1},...,\frac{2u_n}{|u|^2+1},\frac{|u|^2-1}{|u|^2+1}\Big),$$
for $i=1, \ldots, n+1$. Observe that $\Tilde{P_i}$ is a polynomial in $u_1,...,u_n$. So,
$$Q_i=\frac{|u|^2+1}{2(|u|^2+1)^m}(\Tilde{P_i}+u_i\Tilde{P}_{n+1}).$$
Introducing a new independent variable `$s$' so that $ds=\frac{dt}{2(|u|^2+1)^{m-1}}$, we get
$$\frac{du_i}{ds}=\Tilde{P_i}+u_i\Tilde{P}_{n+1}.$$
Hence, the push-forward of the polynomial vector field $\chi=(P_1,...,P_{n+1})$ can be given by the polynomial vector field $\chi'=(R_1,...,R_n)$ under the stereographic projection where $R_i=\Tilde{P_i}+u_i\Tilde{P}_{n+1}$ for $i=1,\ldots,n$.

Note that a vector field on $\SN^n$ is invariant by the special orthogonal group $SO(n+1)$. So, if a polynomial vector field on $\SN^n$ has an invariant $(n-1)$-sphere ${\{\sum\limits_{i=1}^{n+1}a_ix_i+b=0\}\cap \SN^n}$, then we can assume that the $(n-1)$-sphere to be $\{x_{n+1}+b=0\}\cap \SN^n$ with $|b| <1$.

Observe that any $(n-1)$-sphere in $\SN^n$ is connected for $n\geq 2$. Also, an invariant $0$-sphere in $\SN^1$ is the union of two singular points. So, in what follows, we consider $n\geq 2$ whenever we deal with invariant $(n-1)$-sphere in $\SN^n$.
%=============================
%=============================
%=============================

\section{Polynomial vector fields on $\SN^n$}\label{sec:vf_on_sn}

In this section, we characterize any polynomial vector fields on the standard $n$-sphere in $\RR^{n+1}$, discuss their first integrals and complete integrability. We study polynomial vector fields on $\SN^n$ of degree up to two having an invariant great $(n-1)$-sphere in $\SN^n$. We characterize degree one Hamiltonian vector fields on $\SN^{2n-1}$ and give sufficient conditions for higher degree vector fields to be Hamiltonian. 

The following result is algebraic. However, it is crucial in classifying vector fields on $\SN^n$.  
\begin{lemma}\label{lem:sum-pi-xi-0}
    Let $Q_i\in \RR[x_1,...,x_d]$ for $i=1,..., d$ such that the polynomial $\sum\limits_{i=1}^d Q_ix_i^k$ is zero for some $k\geq 1$. Then $Q_i=\sum\limits_{j=1}^d  A_{ij}x_j^k$ for some $A_{ij}\in \RR[x_1,...,x_d]$ such that $A=(A_{ij})_{d\times d}$ is a skew-symmetric matrix.
\end{lemma}
\begin{proof}
   We have $Q_i x_i^k = -\sum\limits_{j=1,j\neq i}^d Q_jx_j^k$.
   There exists a non-negative integer $n_j$ such that we can write 
   $Q_j = \sum\limits_{r=0}^{n_j} \xi_{jr} x_i^r $  for some $\xi_{jr} \in \RR[x_1, \ldots, x_{i-1}, x_{i+1}, \ldots, x_d]$ for $j=1, \ldots, d$. Thus, $$Q_ix_i^k = - \sum_{j=1,j\neq i}^d  \Big( \sum_{r=k}^{n_j} \xi_{jr} x_i^r \Big)x_j^k - \sum_{r=0}^{k-1} \Big( \sum_{j=1,j\neq i}^d \xi_{jr} x_j^k  \Big) x_i^r.$$ Since, $x_1,...,x_d$ are mutually prime, $\sum\limits_{j=1,j\neq i}^d  \xi_{jr} x_j^k=0$. Hence, $Q_i$ can be written as $\sum\limits_{j=1}^d B_{ij}x_j^k$ where $B_{ij}:=-\sum\limits_{r=k}^{n_j} \xi_{jr} x_i^{r-k} \in \RR[x_1,...,x_d]$ if $i\neq j$ and $B_{ii}=0$ for $i=1, \ldots, d$. Such a matrix $(B_{ij})_{d \times d}$ will be called \textit{a matrix for $(Q_1, \ldots, Q_d)$}. 
   
   Now, we show by induction that there exists a skew-symmetric matrix $A=(A_{ij})_{d\times d}$ for $(Q_1, \ldots, Q_d)$. Observe that if $(B_{ij})_{d\times d}$ is a matrix for $(Q_1, \ldots, Q_d)$ then $B_{11}=0$. Assume that there exists a matrix $(C_{ij})_{d \times d}$ for $(Q_1, \ldots, Q_d)$ such that $C_{ij} = - C_{ji}$ for $1 \leq i, j \leq \ell-1 <d$.  
%   Observe that $Q_1=\sum\limits_{j=1}^n A_{1j}x_j^k$. Suppose that up to $\ell-1$-th row the skew-symmetricity is preserved, that is $A_{ij}=-A_{ji}$ for $i,j\leq \ell -1$. We claim that the $\ell$-th row will also preserve the skew-symmetricity.
   We write $C_{\ell 1}=-C_{1\ell}+C'_{\ell 1}$ for some $C'_{\ell 1}\in \RR[x_1,...,x_d]$. Hence, $Q_{\ell}=(-C_{1\ell}+C'_{\ell 1})x_1^k+\sum\limits_{j=2, j\neq \ell}^d C_{\ell j}x_j^k$. Suppose $C'_{\ell 1}\neq 0$. Then $ -C'_{\ell 1} = \alpha x_1^{m_1} \cdots x_d^{m_d} + \phi_{\ell 1}$ for a unique $\phi_{\ell 1} \in \RR[x_1, \ldots, x_d]$ and $\alpha \neq 0$. Then the induction hypothesis and  $\sum\limits_{i=1}^n Q_ix_i^k=0$ implies the following. $$(\alpha x_1^{m_1}...x_d^{m_d} + \phi_{\ell 1})x_1^kx_\ell^k = \sum\limits_{j= \ell+1}^d C_{1 j}x_1^kx_j^k + \sum\limits_{s=2}^{\ell-1} \sum\limits_{j= \ell}^d C_{s j} x_s^kx_j^k + \sum\limits_{j=2}^d C_{\ell j}x_\ell^k x_j^k + \sum\limits_{s=\ell+1}^{d} \sum\limits_{j= 1}^d C_{s j} x_s^kx_j^k.$$ We claim that at least one of $m_2, \ldots, m_{\ell -1}, m_{\ell +1},...,m_d$ should be bigger than or equal to $k$. Suppose that $0 \leq m_2, \ldots, m_{\ell -1}, m_{\ell +1}, ..., m_d < k$. Then the expression of the above identity says that there is no monomial of the form $x_1^{m_1+k} x_2^{m_2} \cdots x_{\ell-1}^{m_{\ell-1}} x_\ell^{m_\ell+k} x_{\ell+1}^{m_{\ell+1}} \cdots x_d^{m_d}$ on the right side of this identity. So, this contradiction proves the claim. Applying the same methods on each monomials of $\phi_{\ell 1}$, one can write $Q_\ell = -C_{1\ell} x_1^k + \sum\limits_{j=2}^d C^1_{\ell j}x_j^k$ for some $C^1_{\ell j} \in \RR[x_1, \ldots, x_d]$ with $C^1_{\ell \ell}=0$. Note that if $C'_{\ell 1}=0$ then $Q_\ell$ is in this form already and $C^1_{\ell j}=C_{\ell j}$ for $j=2, \ldots, d$.
   
   Now we write $C^1_{\ell 2} = -C_{2 \ell}+C'_{\ell 2}$ for some $C'_{\ell 2}\in \RR[x_1,...,x_d]$. Suppose $C'_{\ell 2}\neq 0$. Then $ -C'_{\ell 2} = \beta x_1^{n_1} \cdots x_d^{n_d} + \phi_{\ell 2}$ for a unique $\phi_{\ell 2} \in \RR[x_1, \ldots, x_d]$ and $\beta \neq 0$. We claim that at least one of $n_3, \ldots, n_{\ell -1}, n_{\ell +1},...,n_d$ should be bigger than or equal to $k$. Suppose that $0 \leq n_3, \ldots, n_{\ell -1}, n_{\ell +1}, ..., n_d < k$. Now, the induction hypothesis and $\sum\limits_{i=1}^n Q_ix_i^k=0$ implies the following. $$
   (\beta x_1^{n_1}...x_d^{n_d} + \phi_{\ell 2})x_2^kx_\ell^k = \sum\limits_{s=1}^2 \sum\limits_{j= \ell+1}^d C_{s j}x_s^kx_j^k + \sum\limits_{s=3}^{\ell-1} \sum\limits_{j= \ell}^d C_{s j} x_s^kx_j^k + \sum\limits_{j=3}^d C^1_{\ell j}x_\ell^k x_j^k + \sum\limits_{s=\ell+1}^{d} \sum\limits_{j= 1}^d C_{s j} x_s^kx_j^k.$$
   Then the expression of the above identity says that there is no monomial of the form $x_1^{n_1} x_2^{n_2+k} \cdots x_{\ell-1}^{n_{\ell-1}} x_\ell^{n_\ell+k} x_{\ell+1}^{n_{\ell+1}} \cdots x_d^{n_d}$ on the right side of this identity. So, this contradiction proves the claim. Applying the same methods on each monomials of $\phi_{\ell 2}$, one can write $Q_\ell = -C_{1\ell} x_1^k - C_{2 \ell} x_2^k + \sum\limits_{j=3}^d C^2_{\ell j}x_j^k$ for some $C^2_{\ell j} \in \RR[x_1, \ldots, x_d]$ with $C^2_{\ell \ell}=0$. Note that if $C'_{\ell 2}=0$ then $Q_\ell$ is in this form already, and $C^2_{\ell j}=C^1_{\ell j}$ for $j=3, \ldots, d$.
      
   Proceeding in a similar way, one can write $Q_{\ell}=-\sum\limits_{j=1}^{\ell -1}C_{j\ell}x_j^k + \sum\limits_{j=\ell+1}^n C^{\ell}_{\ell j}x_j^k$. Thus, there exists a matrix $(D_{ij})_{d \times d}$ for $(Q_1, \ldots, Q_d)$ such that $D_{ij} = - D_{ji}$ for $1 \leq i, j \leq \ell \leq d$. Therefore,  
 by induction, there exists a skew-symmetric matrix for $(Q_1, \ldots, Q_d)$.
\end{proof}

\begin{proof}[\textbf{Proof of  \Cref{thm:deg-n-vfld}}]

For each $i \in \{1, \ldots, n+1\}$, we write $P_i= P_i^{(m)}+ \cdots + P_i^{(1)} + P_i^{(0)}$ where $P_i^{(j)}$ is the degree $j$ homogeneous part of $P_i$ for $j =0, \ldots, m$. 
%Similarly, we write $Q =Q^{(3)}+Q^{(2)}+Q^{(1)}+P^{(0)}$ and $R = R^{(3)}+R^{(2)}+R^{(1)}+R^{(0)}$ also in this fashion.
Assume that $\chi=(P_1, \ldots, P_{n+1})$ is a vector field on $\SN^n$. Then, it must satisfy
    \begin{equation}\label{eq:n-deg}
    \sum_{i=1}^{n+1} P_ix_i =K (\sum_{k=1}^{n+1} x_k^2 -1).
\end{equation} for some $K\in \RR[x_1, \ldots, x_{n+1}]$ with $\deg(K)\leq m-1$.
Assume that $K^{(j)}$ is the degree $j$ homogeneous part of $K$. Then $K^{(0)}$ is $0$ since there is no constant term on the left side of \eqref{eq:n-deg}. Now, comparing the degree $(\ell+1)$ terms in \eqref{eq:n-deg}, we obtain 
    \begin{equation}\label{eq:deg-l-poly}
 \sum\limits_{i=1}^{n+1} P_i^{(\ell)}x_i=(\sum\limits_{k=1}^{n+1} x_k^2)K^{(\ell-1)}-K^{(\ell+1)}~;~ \mbox{for } 0\leq \ell\leq m
    \end{equation}
    where $K^{(-1)}=K^{(0)}=K^{(m)}=K^{(m+1)}=0$.
Note that if $m=1$ then $\sum\limits_{i=1}^{n+1}P_ix_i=0$. By Lemma \ref{lem:sum-pi-xi-0}, $P_i=\sum\limits_{j=1}^{n+1}A_{ij}x_j$ where $(A_{ij})$ is a constant skew-symmetric matrix. Next, we assume that $m\geq 2$.

    We claim that for $\ell\geq 1$,
\begin{dmath}\label{eq:p_2l}
        \sum\limits_{i=1}^{n+1} P_i^{(2\ell)}x_i=-(\sum\limits_{k=1}^{n+1} x_k^2)\sum\limits_{i=1}^{n+1} P_i^{(2\ell -2)}x_i-(\sum\limits_{k=1}^{n+1} x_k^2)^2\sum\limits_{i=1}^{n+1} P_i^{(2\ell -4)}x_i-\cdots-(\sum\limits_{k=1}^{n+1} x_k^2)^{\ell}\sum\limits_{i=1}^{n+1} P_i^{(0)}x_i-K^{(2\ell+1)}.
\end{dmath}
    
    We prove the claim inductively. Observe from \eqref{eq:deg-l-poly} that $$\sum\limits_{i=1}^{n+1} P_i^{(2)}x_i=(\sum\limits_{k=1}^{n+1} x_k^2)K^{(1)}-K^{(3)}~\mbox{and}~\sum\limits_{i=1}^{n+1} P_i^{(0)}x_i=-K^{(1)}.$$ Hence, $\sum\limits_{i=1}^{n+1} P_i^{(2)}x_i=-(\sum\limits_{k=1}^{n+1} x_k^2)\sum\limits_{i=1}^{n+1} P_i^{(0)}x_i-K^{(3)}$. So, for $\ell=1$, our claim is true.

    Assume that
    \begin{equation}\label{eq:assumption}
         \sum\limits_{i=1}^{n+1} P_i^{(2\ell-2)}x_i=-(\sum\limits_{k=1}^{n+1} x_k^2)\sum\limits_{i=1}^{n+1} P_i^{(2\ell -4)}x_i-\cdots-(\sum\limits_{k=1}^{n+1} x_k^2)^{\ell -1}\sum\limits_{i=1}^{n+1} P_i^{(0)}x_i-K^{(2\ell-1)}.
    \end{equation}
   We know from \eqref{eq:deg-l-poly} that $\sum\limits_{i=1}^{n+1} P_i^{(2\ell)}x_i=(\sum\limits_{k=1}^{n+1} x_k^2)K^{(2\ell-1)}-K^{(2\ell+1)}.$
So, replacing $K^{(2\ell-1)}$ in this equation from \eqref{eq:assumption}, we get \eqref{eq:p_2l}. Hence, the claim is proved.

Similarly, one can prove that, for $\ell\geq 1$,
\begin{dmath}\label{eq:p_2l-1}
 \sum\limits_{i=1}^{n+1} P_i^{(2\ell+1)}x_i=-(\sum\limits_{k=1}^{n+1} x_k^2)\sum\limits_{i=1}^{n+1} P_i^{(2\ell -1)}x_i-(\sum\limits_{k=1}^{n+1} x_k^2)^2\sum\limits_{i=1}^{n+1} P_i^{(2\ell -3)}x_i-\cdots-(\sum\limits_{k=1}^{n+1} x_k^2)^{\ell}\sum\limits_{i=1}^{n+1} P_i^{(1)}x_i-K^{(2\ell+2)}.
\end{dmath}
We define $d(m):=\begin{cases}
                \frac{m-1}{2}&~\mbox{if}~m~\mbox{is odd,}\\
                \frac{m}{2}&~\mbox{if}~m~\mbox{is even.}
            \end{cases}$
            
Then, using \eqref{eq:p_2l} and \eqref{eq:p_2l-1}, we get the following.
\begin{dmath}
    \sum\limits_{i=1}^{n+1} (P_i^{(m)}+P_i^{(m-1)})x_i=-\sum\limits_{j=1}^{d(m)}(\sum\limits_{k=1}^{n+1} x_k^2)^{j} \sum\limits_{i=1}^{n+1} P_i^{(m-2j)}x_i -\sum\limits_{j=1}^{d(m)} (\sum\limits_{k=1}^{n+1} x_k^2)^{j}\sum\limits_{i=1}^{n+1} P_i^{(m-2j-1)}x_i.
\end{dmath}
where $P^{(-1)}_i:=0$ for $1\leq i\leq n+1$.
Hence, by Lemma \eqref{lem:sum-pi-xi-0},
$$P_i^{(m)}+P_i^{(m-1)}=-\sum\limits_{j=1}^{d(m)}(\sum\limits_{k=1}^{n+1} x_k^2)^{j}  P_i^{(m-2j)} -\sum\limits_{j=1}^{d(m)}(\sum\limits_{k=1}^{n+1} x_k^2)^{j} P_i^{(m-2j-1)}+\sum\limits_{j=1}^{n+1} A_{ij}x_j$$
for each $i\in \{1,...,n+1\}$ where each $A_{ij}$ is a polynomial having monomials of degrees $m-1$ or $m-2$ such that $(A_{ij})$ is a skew-symmetric matrix. So,
\begin{equation*}
    \begin{split}
        P_i=&P_i^{(m)}+P_i^{(m-1)}+\cdots+P^{(0)}\\
        =&-\sum\limits_{j=1}^{d(m)} (\sum\limits_{k=1}^{n+1} x_k^2)^{j} P_i^{(m-2j)} -\sum\limits_{j=1}^{d(m)}(\sum\limits_{k=1}^{n+1} x_k^2)^{j} P_i^{(m-2j-1)}+\sum\limits_{j=1}^{n+1} A_{ij}x_j+\sum\limits_{j=0}^{m-2}P_i^{(j)}\\
        =&\sum\limits_{j=1}^{d(m)}(1-(\sum\limits_{k=1}^{n+1} x_k^2)^{j}))P_i^{(m-2j)}+\sum\limits_{j=1}^{d(m)}(1-(\sum\limits_{k=1}^{n+1} x_k^2)^j)P_i^{(m-2j-1)}+\sum\limits_{j=1}^{n+1} A_{ij}x_j\\
        =&\sum\limits_{j=1}^{d(m)}(1-(\sum\limits_{k=1}^{n+1} x_k^2)^j)(P_i^{(m-2j)}+P_i^{(m-2j-1)})+\sum\limits_{j=1}^{n+1} A_{ij}x_j
    \end{split}
\end{equation*}
Assuming $f_{ij}:=P_i^{(m-2j)}+P_i^{(m-2j-1)}$, we obtain $P_i=\sum\limits_{j=1}^{d(m)}(1-(\sum\limits_{k=1}^{n+1} x_k^2)^j)f_{ij}+\sum\limits_{j=1}^{n+1} A_{ij}x_j$. Observe that $f_{ij}$ is a polynomial having monomials of degrees $m-2j$ or $m-2j-1$. 

Suppose that $P_i$ is given by \eqref{eq:deg-n-form} for $i=1,...,n+1$. Then $P_1, \ldots, P_{n+1}$ satisfy \eqref{eq:n-deg} for some polynomial $K$. Hence, the converse part is also true.
\end{proof}

\begin{remark}
For $i \in \{1, \ldots, n+1\}$, the polynomial $P_i$ in \Cref{thm:deg-n-vfld} can also be written as follows.
    \begin{equation}\label{eq:deg-n-form2}
   P_i=(1-\sum\limits_{k=1}^{n+1} x_k^2) f_{i} +\sum\limits_{j=1}^{n+1} A_{ij}x_j,
    \end{equation}
    for some polynomial $f_i \in \RR[x_1, \ldots, x_{n+1}]$ of degree $\leq m-2$. 
\end{remark}

\begin{theorem}\label{thm:linear-fi}
 Let  $f\in \RR[x_1,\ldots,x_{n+1}]$ be a linear polynomial. Then, $f$ is a first integral of some vector fields on $\SN^n$.
\end{theorem}
\begin{proof}
   Assume that $f:=\sum\limits_{i=1}^{n+1}a_ix_i+c$ where $a_i,c\in \RR$. We construct an $(n+1)\times (n+1)$ skew-symmetric matrix $A=(A_{ij})$ such that $a^tA=0$ and $Ax\neq 0$ where $A_{ij}\in \RR[x_1,\ldots,x_{n+1}]$,  $a^t=(a_1\ldots a_{n+1})$ and $x^t=(x_1\ldots x_{n+1})$ are $1 \times n+1$ matrices. Then, observe that $\chi=(P_1,\ldots,P_{n+1})$ with $P_i=\sum\limits_{i=1}^{n+1}A_{ij}x_j$ is a non-zero vector field on $\SN^n$ such that $f$ is a first integral of $\chi$.

    Suppose $a_k\neq 0$ for some $k\in \{1,\ldots,n+1\}$. Choose an $(n+1)\times (n+1)$ skew-symmetric matrix $B=(B_{ij})$ where $B_{ij}\in \RR[x_1,\ldots,x_{n+1}]$. We define, $$A_{kj}:=-\frac{1}{a_k}\Big(\sum\limits_{i=1,i\neq k}^{n+1}a_iB_{ij}\Big),~A_{jk}:=-A_{kj}~\mbox{for}~j\in \{1,\ldots,k-1,k+1,\ldots,n+1\}.$$
    Also, we define, $A_{pq}:=B_{pq}$ if $p,q\in \{1,\ldots,k-1,k+1,\ldots,n+1\}$ and $A_{kk}=0$. Notice that $A$ is a skew-symmetric matrix. Since $B$ is an arbitrary skew-symmetric matrix, we can assume that $Ax\neq 0$. Observe that $a$ is orthogonal to the $j$-th column of $A$ for $j\neq k$. We show that $a$ is orthogonal to the $k$-th column of $A$ also.
    \begin{equation*}
    \begin{split}
        \sum\limits_{j=1}^{n+1} a_jA_{jk}=-\sum\limits_{j=1,j\neq k}^{n+1} a_jA_{kj}&=\sum\limits_{j=1,j\neq k}^{n+1}\frac{a_j}{a_k}\Big(\sum\limits_{i=1,i\neq k}^{n+1}a_iB_{ij}\Big)\\
        &=\sum\limits_{i=1,i\neq k}^{n+1} \frac{a_i}{a_k}\sum\limits_{j=1,j\neq k}^{n+1}a_jB_{ij}\\
        &=-\sum\limits_{i=1,i\neq k}^{n+1} \frac{a_i}{a_k}\sum\limits_{j=1,j\neq k}^{n+1}a_jB_{ji}\\
        &=-\sum\limits_{i=1,i\neq k}^{n+1} \frac{a_i}{a_k}\sum\limits_{j=1,j\neq k}^{n+1}a_jA_{ji}\\
        &=-\sum\limits_{i=1,i\neq k}^{n+1} \frac{a_i}{a_k}(-a_kA_{ki})\\
        &=-\sum\limits_{i=1}^{n+1}a_iA_{ik}.
    \end{split}
    \end{equation*}
    Hence, $\sum\limits_{j=1}^{n+1}a_jA_{jk}=0$. This completes the proof.
\end{proof}

\begin{remark}
If $n$ is even and $\chi$ is a degree one vector field on $\SN^n$ then $\chi$ has a linear polynomial first integral because the real skew-symmetric matrix of odd order has a real eigenvector corresponding to the eigenvalue 0. 
\end{remark}

\begin{theorem}\label{thm:cmplt-int-vfld}
There exist completely integrable degree $m$ vector fields on $\SN^n$ for each $m,n\in \NN$.
\end{theorem}

\begin{proof}%[\textbf{Proof of \Cref{thm:cmplt-int-vfld}}]
If $A$ is a degree $m-1$ polynomial and $P_1=Ax_2, P_2=-Ax_1, P_3 =0, \ldots, P_{n+1}=0$, then $\chi_n=(P_1, \ldots, P_{n+1})$ is a vector field on $\SN^n$ of degree $m$. Note that $f_1:=\sum\limits_{i=1}^{n+1}x_i^2-1$ is a first integral of $\chi_n$. Also, $f_j:=x_j$ is a first integral of $\chi_n$ for $j=3,\ldots,n+1$. Hence, $\chi_n$ has $n$ independent first integrals, which makes it completely integrable.
\end{proof}

Next, we discuss some necessary conditions for the existence of invariant great $(n-1)$-spheres of a polynomial vector field on $\SN^n$. Note that a vector field on $\SN^n$ is invariant by $SO(n+1)$. Hence, if a polynomial vector field on $\SN^n$ has an invariant great $(n-1)$-sphere, then we can assume the sphere to be $\{x_{n+1}=0\}\cap \SN^n$.

\begin{theorem}\label{thm:inv-grt-sphere-upto-2}
    Let $\chi=(P_1,...,P_{n+1})$ be a degree $m$ vector field on $\SN^n$. Suppose $\chi$ has the invariant great $(n-1)$-sphere $\{x_{n+1}=0\}\cap \SN^n$. The following hold.
    \begin{enumerate}
        \item If $m=1$ then $P_{n+1}=0$.
        \item If $m=2$ then $P_{n+1}=(1-\sum\limits_{k=1}^{n+1} x_k^2) \alpha+\sum\limits_{j=1}^n A_{(n+1)j}x_j$ where $\alpha$ is a real number and $A_{(n+1)j}=c_1^jx_1+\cdots+c_{n+1}^jx_{n+1}$ so that $C=(C_{ij})_{(n+1)\times (n+1)}=(c_i^j)$ is a constant skew-symmetric matrix.
    \end{enumerate}
\end{theorem}
\begin{proof}
    As $\chi=(P_1,...,P_{n+1})$ is a vector field on $\SN^n$, we have $$\sum\limits_{i=1}^{n+1} P_i x_i=0~ \mbox{for all}~ (x_1,\ldots,x_{n+1})\in \SN^n.$$
    Hence, after the stereographic projection, we get
    $$2\sum\limits_{i=1}^n \Tilde{P}_i u_i+(|u|^2-1)\Tilde{P}_{n+1}=0,~\forall ~(u_1,\ldots,u_n)\in \RR^n.$$
    So,
    \begin{equation}\label{eq:sum-u-ui}
        \sum\limits_{i=1}^n \dot{u}_i u_i=\sum\limits_{i=1}^n (\Tilde{P}_i+u_i\Tilde{P}_{n+1}) u_i=\frac{1-|u|^2}{2}\Tilde{P}_{n+1}+|u|^2\Tilde{P}_{n+1}=\frac{1+|u|^2}{2}\Tilde{P}_{n+1}.
    \end{equation}
   The intersection $\{x_{n+1}=0\}\cap \SN^n$ reduces to $\SN^{n-1}$ in $\RR^n$ under the stereographic projection. The sphere $\SN^{n-1}$ is invariant for the vector field $\chi'=(R_1,\ldots,R_n)$ if and only if $|u|^2-1$ divides $\sum\limits_{i=1}^n \dot{u}_i u_i$. By \eqref{eq:sum-u-ui}, $\SN^{n-1}$ is invariant if and only if $|u|^2-1$ divides $\Tilde{P}_{n+1}.$ 

\noindent\textbf{\underline{For (i)}:} If $m=1$, suppose $P_{n+1}=\sum\limits_{j=1}^n A_{(n+1)j}x_j$ where $A_{(n+1)j}$ are constants. Then $\Tilde{P}_{n+1}=2\sum\limits_{i=1}^n A_{(n+1)j}u_j$. Notice that $|u|^2-1$ divides $\Tilde{P}_{n+1}$ if and only if $A_{(n+1)j}=0$ for $j=1,\ldots,n$. Consequently, $\SN^{n-1}$ is invariant if and only if $P_{n+1}=0$. Hence, the first part is proved. 

\noindent\textbf{\underline{For (ii)}:} If $m=2$, suppose $P_{n+1}=(1-\sum\limits_{k=1}^{n+1} x_k^2) \alpha +\sum\limits_{j=1}^n A_{(n+1)j}x_j$ where $A_{(n+1)j}$ are degree one polynomials. We get that $\Tilde{P}_{n+1}=2\sum\limits_{j=1}^n \Tilde{A}_{(n+1)j}u_j$ where $$\Tilde{A}_{(n+1)j}=(|u|^2+1)A_{(n+1)j}\Big(\frac{2u_1}{|u|^2+1},...,\frac{2u_n}{|u|^2+1},\frac{|u|^2-1}{|u|^2+1}\Big).$$ Assume that
    $$A_{(n+1)j}=c_1^jx_1+\cdots+c_{n+1}^jx_{n+1}+c_0^j;~c_i^j\in \RR.$$
    Then $\Tilde{P}_{n+1}|_{|u|^2=1}=2\sum\limits_{j=1}^n(c_1^ju_1+\cdots+c_n^ju_n+2c_0^j)u_j$. Note that $|u|^2-1$ divides $\Tilde{P}_{n+1}$ if and only if $|u|^2-1$ divides $\Tilde{P}_{n+1}|_{|u|^2=1}$. If $|u|^2-1$ divides $\Tilde{P}_{n+1}|_{|u|^2=1}$ then $c_0^j=0$ for $j=1,\ldots,n$ and 
    \begin{equation}\label{eq:skew-symm}
        \sum\limits_{j=1}^n(c_1^ju_1+\cdots+c_n^ju_n)u_j=0.
    \end{equation}
  The equation \eqref{eq:skew-symm} is equivalent to $u^tCu=0$ where $C_{ij}=c_i^j$ and $u^t=\begin{pmatrix}
        u_1&\cdots&u_n
    \end{pmatrix}$. Since $u$ is arbitrary, $C$ is skew-symmetric, see \cite[Chapter I]{elman2008algebraic}.  
\end{proof}

\begin{theorem}\label{thm:hamiltonian-necessary}
Let $\chi=(P_1, \ldots, P_{2n})$ be a vector field on $\SN^{2n-1}$ of degree $m + 2 (>3)$ where $P_i$'s are given by \eqref{eq:deg-n-form2} with $f_i = \sum\limits_{k=1}^{2n} a_{ik} x_k^m +g_i$ for some $g_i \in \RR[x_1, \ldots, x_{2n}]$. If $\chi$ is Hamiltonian then $\deg{g_i} > m-1$ or $ \deg{A_{ij}} > m$ for $i, j \in \{1, \ldots, 2n\}$.
\end{theorem}

\begin{proof}
We may assume that $m+2=\deg P_1 \geq \cdots \geq \deg P_{2n}$. By \eqref{eq:deg-n-form2} we have $$P_i=\bigg(1 - \sum_{j=1}^{2n} x_j^2\bigg) \bigg(\sum_{k=1}^{2n} a_{ik} x_k^m +g_i \bigg) + \sum_{j=1}^{2n} A_{ij}x_j$$
for $i=1, \ldots, 2n$. Suppose $\deg{g_i} \leq m-1$ and $ \deg{A_{ij}}\leq m$ for $i, j \in \{1, \ldots, 2n\}$. Let $H$ be a Hamiltonian for $\chi$. So, $\frac{\partial H}{\partial x_2} = -P_1$ and $\frac{\partial H}{\partial x_1} = P_2$. Thus, we have the following.

\begin{dmath}
 -H =a_{12} \bigg(\frac{ x_2^{m+1}}{m+1} - \frac{x_1^2 x_2^{m+1}}{m+1} - \frac{ x_2^{m+3}}{m+3} - \sum_{j=3}^{2n} \frac{x_j^2 x_2^{m+1}}{m+1} \bigg) + \bigg(1- \sum_{j=1, j\neq 2}^{2n}x_j^2 \bigg) \bigg( \sum_{k=1, k\neq 2}^{2n} a_{1k}x_k^mx_2 \bigg) - \bigg( \sum_{k=1, k \neq 2}^{2n}a_{1k} x_k^m \bigg) \frac{x_2^3}{3} + \phi_1(x_1, \ldots, x_{2n}) + \eta_2(x_1, x_3, \ldots, x_{2n}).
\end{dmath}
and 
\begin{dmath}
 H =a_{21} \bigg( \frac{ x_1^{m+1}}{m+1} - \frac{ x_1^{m+3}}{m+3} - \sum_{j=2}^{2n} \frac{x_j^2 x_1^{m+1}}{m+1} \bigg) + \bigg(1 - \sum_{j=2}^{2n}x_j^2 \bigg) \bigg( \sum_{k= 2}^{2n} a_{2k}x_k^mx_1 \bigg) - \bigg( \sum_{k= 2}^{2n}a_{2k} x_k^m \bigg) \frac{x_1^3}{3} + \phi_2(x_1, \ldots, x_{2n})+ \eta_1(x_2, x_3, \ldots, x_{2n}),
\end{dmath}
for some $\phi_1, \phi_2, \eta_1, \eta_2 \in \RR[x_1, \ldots, x_{2n}]$ with $\deg \phi_1, \deg \phi_2 \leq m+2$ and $\deg \eta_1, \deg \eta_2 \leq m+3$. Observe that $\eta_\ell$ is independent of $x_\ell$ for $\ell=1, 2.$ Comparing the coefficients of degree $m+3$ terms, one gets that $a_{1k}=0$ for $k=1,\ldots,2n$ since $m>1$ . So, we obtain $\deg P_1 < m+2$. Thus, we arrived at a contradiction. 
\end{proof}

\begin{proof}[\textbf{Proof of \Cref{thm:deg-one-hamiltonian}}]
    Suppose $H$ is a Hamiltonian of the degree one vector field $\chi$ on $\SN^{2n-1}$. Then, $H$ is a quadratic homogeneous polynomial in $\RR[x_1,\ldots,x_{2n}]$. So, $H=x^tMx$ where $x^t=\begin{pmatrix}
        x_1 &\ldots &x_{2n}
    \end{pmatrix}$ and $M=(m_{ij})_{2n\times 2n}$ is a symmetric matrix. Note that $\frac{\partial H}{\partial x_k}=e_k^tMx+x^tMe_k=2e_k^tMx$. So,
    $$P_{2i}=\frac{\partial H}{\partial x_{2i-1}}=2e_{2i-1}^tMx~\mbox{and}~P_{2i-1}=-\frac{\partial H}{\partial x_{2i}}=-2e_{2i}^tMx$$
    for each $i\in \{1,\ldots,n\}$. As $P_j=\sum\limits_{k=1}^{2n}a_{jk}x_k$, we obtain $a_{2i,j}=2m_{2i-1,j}$ and $a_{2i-1,j}=-2m_{2i,j}$ for each $i\in \{1,\dots,n\}$ and $j\in \{1,\ldots,2n\}$. Since $M$ is symmetric, $B=2M$ is also a symmetric matrix.

    Conversely, suppose that $B=(b_{ij})$ defined as in the statement is a symmetric matrix. Then, $H=\frac{1}{2}x^tBx$ where $x^t=\begin{pmatrix}
        x_1 &\ldots &x_{2n}\end{pmatrix}$ is a Hamiltonian for $\chi$. 
\end{proof}

\begin{proposition}\label{thm_hamil_d12}
 If $\chi$ is a degree two Hamiltonian vector field on $\SN^3$ then $\chi$ is non-homogeneous. 
\end{proposition}
\begin{proof}
 Suppose $\chi$ is a homogeneous vector field and $H_2$ is a Hamiltonian of it. Then, $H_2$ is a cubic homogeneous polynomial.  Hence,
$$-\frac{\partial H_2}{\partial x_2}=P_1,~\frac{\partial H_2}{\partial x_1}=P_2,~-\frac{\partial H_2}{\partial x_4}=P_3,~\frac{\partial H_2}{\partial x_3}=P_4.$$
Note that the monomials $x_1^2,x_2^2,x_3^2,x_4^2$ are not present in $P_1,P_2,P_3,P_4$, respectively. Hence, the monomials $x_1^2x_2,x_1x_2^2,x_3^2x_4,x_3x_4^2$ are not present in $H_2$. Thus, $H_2$ will be of the following form. 
\begin{dmath*}
H_2=b_1x_1^3+b_2x_1x_3^2+b_3x_1x_4^2+b_4x_1^2x_3+b_5x_1^2x_4+b_6x_1x_2x_3+b_7x_1x_2x_4+b_8x_1x_3x_4+b_9x_2^3+b_{10}x_2x_3^2+b_{11}x_2x_4^2+b_{12}x_2^2x_3+b_{13}x_2^2x_4+b_{14}x_2x_3x_4+b_{15}x_3^3+b_{16}x_4^3.
\end{dmath*}
We solve
 $-x_1\frac{\partial H_2}{\partial x_2}+x_2\frac{\partial H_2}{\partial x_1}-x_3\frac{\partial H_2}{\partial x_4}+x_4\frac{\partial H_2}{\partial x_3}=0$ and get that
 $b_1=\cdots=b_{16}=0.$ The computation is easy to do with some algebraic manipulator software. Hence, $H_2=0$. So, $\chi$ must be non-homogeneous.
\end{proof}

% We remark that any homogeneous vector field on $\SN^n$ has a polynomial first integral, namely $g=\sum\limits_{i=1}^{n+1} x_i^2 -1$, for $n\in \NN$. However, \Cref{thm_hamil_d12} says that there are homogeneous vector fields on $\SN^n$ which are not Hamiltonian.

%%%%%%%%%%%%%%%%%
%%%%%%%%%%%%%%%%%

\section{Invariant meridian and parallel hyperplanes}\label{sec:inv_mer_par}
In this section, we discuss the maximum number of invariant meridian and parallel hyperplanes for a polynomial vector field on $\SN^n$ and obtain some sharp bounds. We also present examples of vector fields on $\SN^n$ with no meridian hyperplanes.

\begin{remark}\label{rmk:bound-mer}
Llibre and Murza \cite{llibre2018darboux} proved that a polynomial vector field $(P_1, \ldots, P_{n+1})$ on $\SN^n$ ($n\geq 2$) with finitely many invariant meridian hyperplanes can have at most $\binom{n-1}{2}(m_1-1)+(\sum\limits_{i=1}^{n-1} m_i)+1$ invariant meridian hyperplanes where $\deg P_i=m_i$ with $m_1\geq \cdots\geq m_{n+1}$ and $\binom{1}{2}=0$. In particular, if $m_i=m$ for $i=1, \ldots, n-1$, then this bound is equal to $\binom{n-1}{2}(m-1)+(n-1)m +1$.
\end{remark}

\begin{proof}[\textbf{Proof of \Cref{mer-bound-reach}} (1)]
 Let $\chi=(P_1, \ldots, P_{n+1})$ be a vector field on $\SN^n$ given by \eqref{eq:deg-n-form} with $f_{ij}=0$ for any $i,j$ and $A_{ij}=f^{m-1}$ for any $1 \leq i < j \leq n+1$ where $f:=\sum\limits_{k=1}^{n} a_kx_k$ and $a_k\in \RR$. Then, $\chi$ is a degree $m$ vector field. Observe that $f^{m-1}$ divides $P_i$ for $1\leq i\leq n+1$. Consider the following extactic polynomial associated with $W=\langle x_1,\ldots,x_n\rangle$.
$$\mathcal{E}_W(\chi)=\begin{vmatrix}
x_1 &  v_2 & \cdots &   x_n \\ P_1 & P_2 & \cdots & P_n \\
\chi(P_1) & \chi(P_2) & \cdots & \chi(P_n)\\
\vdots & \vdots & \ddots  & \vdots \\ \chi^{n-2}(P_1) & \chi^{n-2}(P_2) & \cdots & \chi^{n-2}(P_n) \end{vmatrix}.$$
One can check that if $f^{\alpha}$ divides some polynomial $Q$ then $f^{m-2+\alpha}$ divides $\chi(Q)$. Hence, $f^{m_0+\cdots+m_{n-2}}$ divides $\mathcal{E}_W(\chi)$ where $m_0=m-1$ and $m_{j+1}=m-2+m_j$ for $0\leq j\leq n-3$. Note that $s:=m_0+\ldots+m_{n-2}=\frac{n-1}{2}(2(m-1)+(n-2)(m-2))=\binom{n-1}{2}(m-2)+(n-1)(m-1)$. Hence, $f^{s}$ divides $\mathcal{E}_W(\chi)$. We can see that $f=0$ is an invariant meridian hyperplane of $\chi$. Hence, the result follows.
\end{proof}

\begin{proof}[\textbf{Proof of \Cref{mer-bound-reach}} (2)] Let $\chi=(P_1, \ldots, P_{n+1})$ be a vector field on $\SN^n$ given by \eqref{eq:deg-n-form} with $f_{ij}=0$ for any $i, j$ and $A_{ij}=A \in \RR[x_1, \ldots, x_{n+1}]$ for any $1 \leq i < j \leq n+1$.  Then $A$ is a factor of $\chi^{\ell}(x_j)$ for $1 \leq \ell, j \leq n$. Therefore, $A^{n-1}$ is a factor of the extactic polynomial associated with the vector space generated by $\{x_1, \ldots, x_n\}$ for $\chi$. Let $A:=\prod\limits_{i=1}^{m-1}(d_{i1}x_1 + \cdots + d_{in}x_n)$. Then, for each $i \in \{1, \ldots, m-1\}$, the meridian hyperplane $d_{i1}x_1 + \cdots + d_{in}x_n=0$ is invariant. Thus, the proof follows. \end{proof}

\begin{proof}[\textbf{Proof of \Cref{mer-bound-reach}} (3)] Let $\chi_2 = (P_1, P_2, P_3)$ be a degree $m$ vector field on $\SN^2$. Then, from \eqref{eq:deg-n-form2}, $P_i=(1-\sum\limits_{k=1}^3 x_k^2)f_i+\sum\limits_{j=1}^3 A_{ij}x_j$ for $i=1,2,3$, where $(A_{ij})_{3 \times 3}$ is a skew-symmetric matrix. So, the extactic polynomial associated with $\left< x_1,x_2\right>$ is $$(1-\sum\limits_{k=1}^3x_k^2)(f_2x_1-f_1x_2)-A_{12}(x_1^2+x_2^2) +(A_{23}x_1-A_{13}x_2)x_3$$ of degree $m+1$. This can be expressed as $(x_1^2+x_2^2)f +x_3g+h$ for some $f \in  \RR[x_1, x_2]$ and $g,h \in \RR[x_1, x_2, x_3]$ with $\deg h\leq m-1$. Suppose that $\chi_2$ has $m+1$ invariant meridian hyperplane. Then, $(x_1^2+x_2^2)f +x_3g+h$ is a homogeneous polynomial in $\RR[x_1, x_2]$. Thus, $g=h=0$. Since $x_1^2+x_2^2$ is irreducible in $\RR[x_1, x_2]$, we get a contradiction. Therefore, the first claim follows.

Suppose that $(P_1, P_2, P_3)$ is given by \eqref{eq:deg-n-form} with $f_{ij}=0$ for all $1 \leq i,j \leq 3$ and $A_{ij}=0$ whenever $(i,j) \neq (2,3)$. Then the corresponding extactic polynomial is $A_{23}x_1x_3$. Assume that $A_{23} = \prod\limits_{i=1}^{m-1}(d_{i1}x_1 +d_{i2}x_2)$. Then, $x_1=0$ and $d_{i1}x_1 +d_{i2}x_2=0$ for $i=1,\ldots,m-1$ are invariant meridian hyperplanes. Hence, we achieve the bound.
\end{proof}

\begin{proof}[\textbf{Proof of \Cref{mer-bound-reach}} (4)] Consider the vector field $\chi_3=(P_1,\ldots, P_4)$ in $\RR^4$ with $$P_1=Ax_2,~P_2=-Ax_1, P_3=P_4=0.$$ Then $\chi_3$ is a vector field on $\SN^3$. Observe that $\mathcal{E}_{\langle x_1, x_2,x_3\rangle}(\chi_3)=-A^3x_3(x_1^2+x_2^2)$. Assume that $A=\prod\limits_{i=1}^{m-1} (a_ix_1+b_ix_2+c_ix_3)$. Note that for each $i\in \{1,\ldots,m-1\}$, the meridian hyperplane $a_ix_1+b_ix_2+c_ix_3=0$ is invariant with multiplicity 3. Also, the hyperplane $x_3=0$ is invariant with multiplicity 1. Thus, $\chi_3$ has $3m-2$ invariant meridian hyperplanes taking into account multiplicities.
\end{proof}
 \begin{remark}
     Llibre and Murza \cite{llibre2018darboux} provided a complex polynomial vector field on the complex sphere for which the corresponding bound in Remark \ref{rmk:bound-mer} is attained for $n=2$. However, for a real polynomial vector field on $\SN^2$, the bound is not sharp.
 \end{remark}

\begin{theorem}\label{thm:no-meridians}
There are polynomial vector fields on $\SN^n$ having no invariant meridian hyperplane.
\end{theorem}
\begin{proof}
    Suppose $P_1,\ldots,P_{2n}\in \RR[x_1,\ldots,x_{2n}]$ such that $P_{2i-1}=A_ix_{2i}$ and $P_{2i}=-A_ix_{2i-1}$ where $A_i$ are non-zero real numbers for $i=1,\ldots,n$. One can check that $\chi=(P_1,\ldots,P_{2n-1})$ is a vector field on $\SN^{2n-1}$. Now, assume that $\sum\limits_{i=1}^{2n-1}a_ix_i=0$ is an invariant meridian hyperplane of $\chi$. Then, there exists $k_0\in \RR$ such that
         $$   \chi(\sum\limits_{i=1}^{2n-1}a_ix_i) =k_0(\sum\limits_{i=1}^{2n-1}a_ix_i)
            \implies \sum\limits_{i=1}^{2n-1}a_iP_i =k_0(\sum\limits_{i=1}^{2n-1}a_ix_i).$$
 
    Equating the coefficients of $x_1$ and $x_2$ on both sides, we get $-a_2A_1=k_0a_1$ and $a_1A_1=k_0a_2$, respectively. This gives $a_1=a_2=0$. Similarly, we obtain $a_3=\cdots=a_{2n-2}=0.$ Also, equating coefficient of $x_{2n}$ both side, we get $a_{2n-1}=0$. Hence, we cannot have any invariant meridian hyperplane for $\chi$.

    Similarly, we consider the vector field $\chi'=(Q_1,\ldots,Q_{2n+1})$ on $\SN^{2n}$ such that $Q_{2i-1}=B_ix_{2i}$, $Q_{2i}=-B_ix_{2i-1}$ and $Q_{2n+1}=0$ where $B_i$ are non-zero real numbers for $i=1,\ldots,n$. One can check that $\chi'$ has no invariant meridian hyperplane using similar arguments for the case of $\SN^{2n-1}$.
\end{proof}

% \begin{remark}
Llibre and Murza \cite{llibre2018darboux} showed that a degree $m$ vector field on $\SN^n$ can have $m$ invariant parallel hyperplanes. However, that bound is not sharp.

\begin{proof}[\textbf{Proof of \Cref{thm:parallel-bound}}]
    Assume that $\chi=(P_1,\ldots,P_{n+1})$. Suppose that $\chi$ has $m$ invariant parallel hyperplanes. Then, $P_{n+1}=c\prod\limits_{i=1}^{m}(x_{n+1}-k_i)$ for some $k_i,c\in \RR$ with $c\neq 0$ and $-1<k_i< 1$. Moreover, from \eqref{eq:deg-n-form2}, we have $P_{n+1}=(1-\sum\limits_{k=1}^{n+1} x_k^2)f_{n+1}+\sum\limits_{j=1}^{n}A_{(n+1)j}x_j$ for some $f_{n+1}\in \RR[x_1,\ldots,x_{n+1}]$. Hence, 
    \begin{equation}\label{eq:m-parallel}
        (1-\sum\limits_{k=1}^{n+1} x_k^2)f_{n+1}+\sum\limits_{j=1}^{n}A_{(n+1)j}x_j=c\prod\limits_{i=1}^{m}(x_{n+1}-k_i).
    \end{equation}
    The equation \eqref{eq:m-parallel} is true for all $(x_1,\ldots,x_{n+1})\in \RR^{n+1}$. In particular, for $x_1=\cdots=x_n=0$, \eqref{eq:m-parallel} gives the following.
    $$(1-x_{n+1}^2)f_{n+1}(0,\ldots,0,x_{n+1})=c\prod\limits_{i=1}^{m}(x_{n+1}-k_i).$$
    Hence, $1-x_{n+1}^2$ divides $P_{n+1}$. This contradicts the assumption. Therefore, $\chi$ has at most $m-1$ invariant parallel hyperplanes.

    Consider the vector field $\chi=(P_1,\ldots,P_{n+1})$ of the form \eqref{eq:deg-n-form2} with $f_{n+1}=0$ and $A_{(n+1)j}=\prod\limits_{i=1}^{m-1}(x_{n+1}-k_i)$  for some $k_i\in \RR$. Then, $P_{n+1}=(x_1+\ldots+x_n)\prod\limits_{i=1}^{m-1}(x_{n+1}-k_i)$. Hence, $\chi$ is a degree $m$ vector field on $\SN^n$ and it has $m-1$ invariant parallel hyperplanes. Thus, the bound is reached.
\end{proof}

%%%%%%%%%%%%%%%%%%%%%%%%%%%%
%%%%%%%%%%%%%%%%%%%%%%%%%%%%

\section{Homogeneous polynomial vector fields on $\SN^n$}\label{sec:hom_vf}
In this section, we study homogeneous polynomial vector fields on $\SN^n$ and their invariant $(n-1)$-sphere. Llibre and Pessoa \cite{llibre2006homogeneous} described the equation of the cone on a circle in $\SN^2$. We generalize that to find the equation of the cone on an arbitrary $(n-1)$-sphere in $\SN^n$. Note that the authors and Benny studied the number of great circles for homogeneous polynomial vector fields on $\SN^2$ in \cite{BJS_24}.

\begin{proposition}\label{prop:cone}
    Let $\mathcal{S}:=\{\sum\limits_{i=1}^{n+1} a_ix_i=d\}\cap \SN^n$ be an $(n-1)$-sphere in $\SN^n$ and $$C(\mathcal{S}):=\{(x_1,...,x_{n+1})\in \RR^{n+1} ~|~(\sum\limits_{i=1}^{n+1} a_ix_i)^2-d^2(\sum\limits_{i=1}^{n+1} x_i^2)=0\}.$$ Then $C(\mathcal{S})=\{s(x_1,...,x_{n+1})~|~(x_1,...,x_{n+1})\in \mathcal{S},s\in \RR\}$.
\end{proposition}
\begin{proof}
    Suppose that $(x_1,...,x_{n+1})\in \mathcal{S}$. Then, $\sum\limits_{i=1}^{n+1} a_ix_i=d$ and $\sum\limits_{i=1}^{n+1} x_i^2=1$. Hence, $(\sum\limits_{i=1}^{n+1} a_ix_i)^2-d^2(\sum\limits_{i=1}^{n+1} x_i^2)=0$. Also, $(\sum\limits_{i=1}^{n+1} a_isx_i)^2-d^2(\sum\limits_{i=1}^{n+1} s^2x_i^2)=0$ for any $s\in \RR$. So, $\{s(x_1,...,x_{n+1}):(x_1,...,x_{n+1})\in \mathcal{S},s\in \RR\} \subseteq C(\mathcal{S})$.

    Conversely, suppose that $(y_1,...,y_{n+1})\in C(\mathcal{S})$. Then $(\sum\limits_{i=1}^{n+1} a_iy_i)^2-d^2(\sum\limits_{i=1}^{n+1} y_i^2)=0$. So, we obtain $(\sum\limits_{i=1}^{n+1} a_i\frac{y_i}{|y|})^2=d^2$ where $|y|:=\sqrt{\sum\limits_{i=1}^{n+1}y_i^2}$. Hence, $\sum\limits_{i=1}^{n+1} a_i\frac{y_i}{|y|}=\pm d$. Without loss of generality, assume that $\sum\limits_{i=1}^{n+1} a_i\frac{y_i}{|y|}= d$. Now, assuming $(x_1,...,x_{n+1}):=\frac{1}{|y|}(y_1,...,y_{n+1})$, we find $(x_1,...,x_{n+1})\in \mathcal{S}$ and $(y_1,...,y_{n+1})=|y|(x_1,...,x_{n+1})$. Thus, the result is proved.
\end{proof}
The hypersurface $\mathcal{C}(\mathcal{S})$ is said to be the \textit{cone on $\mathcal{S}$}.
\begin{lemma}\label{lem:cone-invariant}
   Suppose that $\chi$ is a homogeneous polynomial vector field on $\SN^n$. Then, the set $\mathcal{S}$ is invariant if and only if the set $\mathcal{C}(\mathcal{S})$ is invariant for the vector field $\chi$.
\end{lemma}
\begin{proof}
   Suppose that $\mathcal{S}$ is invariant for the vector field $\chi:=(P_1,\ldots,P_{n+1})$ and $\deg \chi=m$. So, $\chi$ is a vector field on $\mathcal{S}$. We show that $\chi$ is also a vector field on $\mathcal{C}(\mathcal{S})$. Take $y\in \mathcal{C}(\mathcal{S})$. Then, $y=sx$ for some $s\in \RR$ and $x\in \mathcal{S}$. Let $\phi(t,x)$ be the integral curve of $\chi$ through $x$. We define, $\psi:\RR\to \mathcal{C}(\mathcal{S})$ given by $\psi(t)=s\phi(t,x)$. Observe that $\psi(0)=y$. Also,
   \begin{equation*}
   \begin{split}
       \frac{d \psi}{dt}\Big|_{t=0}=s\frac{d \phi}{dt}\Big|_{t=0}=s(P_1(x),\ldots,P_{n+1}(x))&=\frac{1}{s^{m-1}}(P_1(sx),\ldots,P_{n+1}(sx))\\
       &=\frac{1}{s^{m-1}}(P_1(y),\ldots,P_{n+1}(y)).
       \end{split}
    \end{equation*}
Therefore, $\chi$ is a vector field on $\mathcal{C}(\mathcal{S})$.

Conversely, assume that $\mathcal{C}(\mathcal{S})$ is invariant for the vector field $\chi$ on $\SN^n$. Then, $$\mathcal{C}(\mathcal{S})\cap \SN^n=\Big\{\sum\limits_{i=1}^{n+1} a_ix_i=\pm d\Big\}\cap \SN^n = -\mathcal{S} \sqcup \mathcal{S}$$ is invariant. Therefore, $\mathcal{S}$ is also invariant, since a trajectory is path-connected.
\end{proof}
\begin{remark}
    One can also check that if $\phi(t,x)$ is the flow of $\chi$ on $\SN^n$ then $s\phi(s^{m-1}t,x)$ is the flow of $\chi$ on the $n$-dimensional sphere of radius $s$.
\end{remark}

\begin{theorem}
\label{prop:hom-classify}
The polynomial vector field $(P_1, \ldots, P_{n+1})$  on $\SN^n$ is homogeneous of degree $m$ if and only if $P_i = \sum\limits_{j=1}^{n+1} A_{ij} x_j$ for $i=1, \ldots, n+1$, where $(A_{ij})$ is a skew-symmetric matrix and $A_{ij}$ is either zero or a homogeneous polynomial of degree $(m-1)$.
\end{theorem}
\begin{proof}
    Suppose that $\chi$ is a degree $m$ homogeneous polynomial vector field on $\SN^n$. Thus, we have $\chi(\sum\limits_{i=1}^{n+1}x_i^2-1)=2\sum\limits_{i=1}^{n+1}P_ix_i=0$ on $\SN^n$. By Lemma \ref{lem:cone-invariant}, $\chi$ is a homogeneous polynomial vector field on the $n$-dimensional sphere of arbitrary radius with origin at $(0,\ldots,0)\in \RR^{n+1}$. So, $\sum\limits_{i=1}^{n+1}P_ix_i=0$ for any $(x_1,\ldots,x_{n+1})\in \RR^{n+1}$. Hence, the proof follows from Lemma \ref{lem:sum-pi-xi-0}.

    If $P_i = \sum\limits_{j=1}^{n+1} A_{ij} x_j$ for $i=1, \ldots, n+1$, where $(A_{ij})$ is a skew-symmetric matrix with each $A_{ij}$ is either zero or a homogeneous polynomial of degree $(m-1)$, then one can verify that $\chi=(P_1,\ldots,P_{n+1})$ is a degree $m$ vector field on $\SN^n$.
\end{proof}

We remark that one can prove \Cref{prop:hom-classify} using \Cref{thm:deg-n-vfld} and some algebraic calculation. However, our proof is shorter and geometric.

Next we look at the invariant $(n-1)$-spheres in $\SN^n$. \Cref{lem:cone-invariant} states that an $(n-1)$-sphere in $\SN^n$ is invariant if and only if the entire cone on the $(n-1)$-sphere is invariant. Suppose $\mathcal{S}:=\{\sum\limits_{i=1}^{n+1} a_ix_i+d=0\}\cap \SN^n$ is an $(n-1)$-sphere in $\SN^n$. Note that the equation of the cone on $\mathcal{S}$ is 
$$(\sum\limits_{i=1}^{n+1} a_ix_i)^2-d^2(\sum\limits_{i=1}^{n+1} x_i^2)=0~\mbox{if $d\neq 0$}.$$
If $d=0$, then the equation of the cone is $\sum\limits_{i=1}^{n+1} a_ix_i=0$, i.e., the cone is the $(n-1)$-sphere itself.

\begin{proof}[\textbf{Proof of \Cref{thm_sp_grtsp}}]
    Without loss of generality, suppose that $\{x_{n+1}+d=0\}\cap \SN^n$ is an invariant $(n-1)$-sphere for $\chi$. It is enough to prove that $d=0$.
    By Lemma \ref{lem:cone-invariant}, $\mathcal{S}=\{x_{n+1}+d=0\}\cap \SN^n$ is invariant if and only if $\mathcal{C}(\mathcal{S})=\{x_{n+1}^2-d^2(\sum\limits_{i=1}^{n+1}x_i^2)=0\}$ is invariant. If $\mathcal{C}(\mathcal{S})$ is invariant, there exists $K\in \RR[x_1,\ldots,x_{n+1}]$ such that
    \begin{equation*}
        \chi(x_{n+1}^2-d^2(\sum\limits_{i=1}^{n+1}x_i^2))=K(x_{n+1}^2-d^2(\sum\limits_{i=1}^{n+1}x_i^2))
    \end{equation*}
Hence, we have
$$2P_{n+1} x_{n+1}=K(x_{n+1}^2-d^2(\sum\limits_{i=1}^{n+1}x_i^2)),$$
since $\chi$ is homogeneous. Recall that $P_{n+1}$ does not contain any monomial of the form $x_{n+1}^{\alpha}$. Therefore, $K$ also does not contain any monomial of the form $x_{n+1}^{\alpha}$. Assume that $d\neq 0$. Then, $x_{n+1}$ must divide $K$. However, $\deg \chi\leq 2$ and consequently $\deg(K)\leq 1$. So, $K$ must be zero. This implies that $P_{n+1}=0$. This contradicts the fact that $\chi$ is homogeneous. Thus, $d$ must be zero.
\end{proof}

\begin{proposition}\label{prop:inv-sp-not-grt}
Let $\chi=(P_1,...,P_{n+1})$ be a homogeneous degree $m\geq 3$ vector field on $\SN^n$. Suppose that $\chi$ has an invariant $(n-1)$-sphere $\{x_{j}+d=0\}\cap \SN^n$ with $0 < |d| <1$. Then $P_{j}=(\sum\limits_{i=1,i\neq j}^{n+1} B_ix_i)(x_{j}^2-d^2(\sum\limits_{i=1}^{n+1}x_i^2))$ where $B_i$ are polynomials of degree at most $m-3$.    
\end{proposition}
\begin{proof}
    Since $\chi = (P_1, \ldots, P_{n+1})$ is homogeneous, by \Cref{prop:hom-classify}, $\sum\limits_{i=1}^{n+1}P_ix_i=0$. By the hypothesis, $\{x_{j}+d=0\}\cap \SN^n$ is invariant under $\chi$ with $d\neq 0$. Then, by Lemma \ref{lem:cone-invariant}, $f:=x_{j}^2-d^2(\sum\limits_{i=1}^{n+1}x_i^2)$ is also invariant. So, for some $K\in \RR[x_1,...,x_{n+1}]$, we have
 $$        \chi f=Kf 
            \implies  2P_{j}x_{j}=K(x_{j}^2-d^2(\sum\limits_{i=1}^{n+1}x_i^2)).$$
From the above, we get $x_{j}$ divides $K$, say $K=2K'x_{j}$. Hence,
$$P_{j}=K'(x_{j}^2-d^2(\sum\limits_{i=1}^{n+1}x_i^2)).$$
Observe that $K'$ should not have any monomial of the form $x_{j}^{\alpha}$ since $P_{j}$ does not have any such monomial. Therefore, $K'=\sum\limits_{i=1,i\neq j}^{n+1} B_ix_i$ where $B_i$ are polynomials of degree at most $m-3$, for $i\in \{1,\ldots,j-1,j+1,\ldots,n+1\}$. Consequently, we obtain that ${P_{j}=(\sum\limits_{i=1,i\neq j}^{n+1} B_ix_i)(x_{j}^2-d^2(\sum\limits_{i=1}^{n+1}x_i^2))}$.
\end{proof}

\begin{proof}[\textbf{Proof of \Cref{thm_inv_grt_deg1}}]
    Let $\chi=(P_1,\ldots,P_n)$ be a degree one vector field on $\SN^n$. By \Cref{thm:deg-n-vfld}, $P_i=\sum\limits_{j=1}^{n+1}A_{ij}x_j$ where $A=(A_{ij})_{(n+1)\times (n+1)}$ is a constant real skew-symmetric matrix. Suppose $\mathcal{S}:=\{\sum\limits_{i=1}^{n+1}a_ix_i=0\}\cap \SN^n$ is an invariant great $(n-1)$-sphere for $\chi$. By \Cref{lem:cone-invariant}, $\mathcal{S}$ is invariant if and only if $\mathcal{C}(\mathcal{S})$ is invariant. So, it is enough to find the maximum number of invariant $\mathcal{C}(\mathcal{S})=\{\sum\limits_{i=1}^{n+1}a_ix_i=0\}$. Suppose that $\mathcal{C}(\mathcal{S})$ is invariant. Then, there exists $k_0\in \RR$ such that 
    \begin{equation*}
        \begin{split}
            \sum\limits_{i=1}^{n+1}a_iP_i=k_0(\sum\limits_{i=1}^{n+1}a_ix_i)
            \implies a^tAx=k_0a^tx
        \end{split}
    \end{equation*}
    where $a^t=\begin{pmatrix}
        a_1&\cdots&a_{n+1}
    \end{pmatrix}$ and $x^t=\begin{pmatrix}
        x_1&\cdots&x_{n+1}
    \end{pmatrix}$. The equality is true for every $x^t\in \RR^{n+1}$. Hence, $a^tA=k_0a^t$, which gives $Aa=-k_0a$. So, $a$ is an eigenvector of $A$ corresponding to $-k_0$. Note that 0 is the only real eigenvalue of $A$ since it is a skew-symmetric matrix. Hence, $\{\sum\limits_{i=1}^{n+1}a_ix_i=0\}\cap \SN^n$ is invariant if and only if $a=\begin{pmatrix}
        a_1&\cdots&a_{n+1}
    \end{pmatrix}^t$ is an eigenvector of $A$ corresponding to the eigenvalue 0.
\end{proof}
\begin{remark}
The set $\{\sum\limits_{i=1}^{n+1}a_ix_i=0\}\cap \SN^n$ is an invariant great $(n-1)$-sphere of $\chi$ if and only if $\sum\limits_{i=1}^{n+1}a_ix_i$ is a first integral of $\chi$. So, a degree one vector field on $\SN^n$ with finitely many invariant great $(n-1)$-spheres can have at most one invariant great $(n-1)$-sphere.
\end{remark}

Let $\{e_1, \ldots, e_{n+1}\}$ be the standard basis of $\RR^{n+1}$ and $X^t:=(x_1, \ldots, x_{n+1})$.

\begin{proof}[\textbf{Proof of \Cref{thm_inv_grt_deg2}}]
Proposition \ref{prop:cone} and Lemma \ref{lem:cone-invariant} give that the set $\{{\bf a}^t X =0\} \cap \SN^n$  is invariant if and only if ${\bf a}^t X =0$ is invariant. Thus, the set $\{{\bf a}^t X =0\} \cap \SN^n$  is an invariant great $(n-1)$-sphere if and only if $\sum\limits_{i=1}^{n+1} a_i P_i = (X^t {\bf b})({\bf a}^t X)$ for some ${\bf b}^t\in \RR^{n+1}$. Notice that 
\begin{equation*}
    \begin{split}
        &X^t(\sum\limits_{i=1}^{n+1}a_i B_i - {\bf b}{\bf a}^t)X =0\\
        \iff & (\sum\limits_{i=1}^{n+1}a_i B_i - {\bf b}{\bf a}^t)^t=-(\sum\limits_{i=1}^{n+1}a_i B_i - {\bf b}{\bf a}^t)~\mbox{ since $X$ is arbitrary, see \cite[Chapter I]{elman2008algebraic}}\\
        \iff &\sum\limits_{i=1}^{n+1}a_i B_i - {\bf a}{\bf b}^t=-(\sum\limits_{i=1}^{n+1}a_i B_i - {\bf b}{\bf a}^t)\\
        \iff&2\sum\limits_{i=1}^{n+1}a_iB_i={\bf a}{\bf b}^t+{\bf b}{\bf a}^t.
    \end{split}
\end{equation*}
Hence, the first claim is proved.

Now, suppose that ${\bf a}$ is an eigenvector of each $A_i$ with eigenvalue $\lambda_i$. So, $A_i{\bf a}=\lambda_i{\bf a}$. Again,
\begin{equation*}
    \begin{split}
        \chi({\bf a}^tX)=X^t(\sum\limits_{i=1}^{n+1}a_iB_i)X=X^t[(A_1{\bf a})\cdots (A_{n+1}{\bf a})]X=X^t[(\lambda_1{\bf a})\cdots (\lambda_{n+1}{\bf a})]X=X^t{\bf a}{\Lambda}^tX
    \end{split}
\end{equation*}
where $\Lambda^t=(\lambda_1,\ldots,\lambda_{n+1})$. Hence, ${\bf a}^tX=0$ produces an invariant great $(n-1)$-sphere with cofactor $\Lambda^tX$.
\end{proof}

\noindent {\bf Acknowledgments.}
The first author is supported by the Prime Minister's Research Fellowship, Government of India. The second author thanks the `ICSR office at IIT Madras' for the SEED research grant, and IMSc Chennai for a month visiting position where a part of the work was completed.

\bibliography{Ref_InvariantAlgebraicSets}
\bibliographystyle{abbrv}

\end{document}